\documentclass{amsart}
\usepackage[utf8]{inputenc}
\usepackage{amsthm}
\usepackage{amsmath}

\DeclareMathOperator\dist{dist}

\usepackage{enumerate}   
\usepackage{amssymb}
\usepackage{fancyhdr} 
\usepackage[left=3.5cm,right=3.5cm,top=3.5cm,bottom=3.5cm]{geometry}
\usepackage{xcolor}
\newcommand{\norm}[1]{\left\lVert#1\right\rVert}
\usepackage{graphicx}
\newtheorem{theorem}{Theorem}[section]

\newtheorem{corollary}[theorem]{Corollary}
\newtheorem{lemma}[theorem]{Lemma}
\theoremstyle{definition}

\theoremstyle{remark}
\newtheorem*{remark}{Remark}
\renewcommand\Re{\operatorname{Re}}
\renewcommand\Im{\operatorname{Im}}
\newcommand\TT{\mathbb T}
\newcommand\veps{\varepsilon}
\newcommand\GHZ{\mathfrak{G}_0}
\newcommand\GHU{\mathfrak{G}_{\geq 1}}
\newcommand\Cud{\mathbb{C}_{\frac{1}{2}}}
\newcommand\Rewud{\Re(w)-\frac 12}
\newcommand\PRewud{\left(\Re(w)-\frac 12\right)}
\newcommand\CC{\mathbb C}
\newcommand\DD{\mathbb D}

\usepackage{url}
\usepackage[hidelinks]{hyperref}
\hypersetup{
	colorlinks=true,
	linkcolor=cyan,
	filecolor=cyan,
	citecolor =cyan,      
	urlcolor=cyan,
}
\author{Frédéric Bayart}
\address{Université Clermont Auvergne, LMBP, Campus des Cézeaux, F-63177 Aubiere Cedex, France}
\email{frederic.bayart@uca.fr}
\author{Athanasios Kouroupis}
\address{Department of Mathematical Sciences, Norwegian University of Science and Technology (NTNU), 7491 Trondheim, Norway}
\email{athanasios.kouroupis@ntnu.no}
\title[Schatten class composition operators]{Schatten class composition operators on the Hardy space of Dirichlet series and a comparison-type principle}
\date{}
\begin{document}
\maketitle
\begin{abstract}
We give necessary and sufficient conditions for a composition operator with Dirichlet series symbol to belong to the Schatten classes $S_p$ of the Hardy space $\mathcal{H}^2$ of Dirichlet series.
For $p\geq 2$, these conditions lead to a characterization for the subclass of symbols with bounded imaginary parts.
Finally, we establish a comparison-type principle for composition operators. 
Applying our techniques in conjunction with classical geometric function theory methods,  we prove the analogue of the polygonal compactness theorem for $\mathcal{H}^2$ and we give examples of bounded composition operators with Dirichlet series symbols on $\mathcal{H}^p,\,p>0$.
\end{abstract}

\section{Introduction}
The Hardy space $\mathcal{H}^2$ of Dirichlet series, which was first systematically studied by H. Hedenmalm, P. Lindqvist, and K. Seip \cite{HLS97}, is defined as
$$\mathcal{H}^2=\left\{f(s)=\sum_{n\geq1}\frac{a_n}{n^s}:\norm{f}_{\mathcal{H}^2}^2=\sum_{n\geq1}|a_n|^2<\infty \right\}.$$

Gordon and Hedenmalm \cite{GH99} determined the class $\mathfrak{G}$ of symbols which generate bounded composition operators on the Hardy space $\mathcal{H}^2$.  The Gordon--Hedenmalm class $\mathfrak{G}$ consists of all functions $\psi(s) = c_0s+\varphi(s)$, where $c_0$ is a non-negative integer, called the characteristic of $\psi$,
and $\varphi$ is a Dirichlet series such that:
\begin{enumerate}[(i)]
	\item If $c_0=0$, then $\varphi(\mathbb{C}_0)\subset\mathbb{C}_\frac{1}{2}$.\label{item1}
	\item If $c_0\geq 1$, then $\varphi(\mathbb{C}_0)\subset\mathbb{C}_0$ or $\varphi\equiv i\tau$ for some $\tau\in\mathbb{R}$.\label{item2}
\end{enumerate}
We denote by  $\mathbb{C}_\theta,\,\theta\in\mathbb{R}$ the half-plane $\{s:\Re s>\theta\}$. We will also use the notation $\mathfrak{G}_0$ and $\mathfrak{G}_{\geq 1}$ for the subclasses of symbols that satisfy (\ref{item1}) and (\ref{item2}), respectively.

In this paper, we are mostly interested in the case $\psi=\varphi\in\GHZ.$ In that context, the compact operators $C_\varphi \colon \mathcal{H}^2 \to \mathcal{H}^2$ were characterized only very recently in \cite{BP21}, in terms of the behavior of the mean counting function
$$M_{\varphi}(w) = \lim_{\sigma \to 0^+}\lim_{T \to \infty}  \frac{\pi}{T}\sum\limits_{\substack{s\in\varphi^{-1}(\{w\})\\
		|\Im s|<T\\
		\sigma<\Re s<\infty}} \Re s ,\qquad w\in\Cud\backslash\{\varphi(+\infty)\}.$$
It turns out that $C_\varphi$ is compact if and only if
\begin{equation}\label{eq:compact}
\lim\limits_{\Re w\rightarrow\frac{1}{2}^+}\frac{M_\varphi(w)}{\Re w-\frac{1}{2}}=0.
\end{equation}
The next step would be to characterize symbols $\varphi\in \GHZ$ such that $C_\varphi$ belongs to the Schatten class $S_p,$ $p>0.$ 
In the disk setting  D. H. Luecking and K. Zhu \cite{LZ92} proved that a composition operator $C_\phi$ on the Hardy space $H^2(\mathbb{D})$ belongs to the Schatten class $S_p,\, p>0$ if and only if 
\begin{equation}
\int\limits_{\mathbb{D}}\frac{\left(N_\phi(z)\right)^{\frac{p}{2}}}{\left(1-|z|^2\right)^{\frac{p}{2}+2}}\, dA(z)<+\infty,
\end{equation}
where $\phi$ is a holomorphic self-map of the unit disk and $N_\phi$ is the associated Nevanlinna counting function \cite{SHAP87}.

Our first main result is that the analogue characterization holds in the Dirichlet series setting
provided the symbol has bounded imaginary part. 

\begin{theorem}\label{boundedchar}
Suppose that the symbol $\varphi\in\mathfrak{G}_0$ has bounded imaginary part and that $p\geq 1$. Then, the composition operator $C_\varphi$ belongs to the class $S_{2p}$ if and only
if $\varphi$ satisfies the condition
\begin{equation}\label{eq:main1}
\int\limits_{\mathbb{C}_\frac{1}{2}}\frac{\left(M_\varphi(w)\right)^{p}}{\left(\Re w-\frac{1}{2}\right)^{p+2}}\, dA(w)<+\infty.
\end{equation}
For $p>0$ the above condition remains necessary and if  $p\geq 2$ it is necessary for all symbols in $\mathfrak{G}_0$.
\end{theorem}

When $p=1,$ namely if we want to know if $C_\varphi$ is Hilbert-Schmidt, things are easier and Hilbert-Schmidt composition operators with symbols $\varphi$ in $\GHZ$ have already
been characterized in \cite{BP21}. This is equivalent to saying that 
$$\int_{\Cud} \zeta''(2\Re(w))M_\varphi(w)dA(w)<+\infty.$$
We generalize this characterization for $C_\varphi\in S_{2m},$ $m\in\mathbb N.$

\begin{theorem}\label{main5}
Let $\varphi\in \mathfrak{G}_0$ and $m\in\mathbb N.$ Then $C_\varphi$ belongs to $S_{2m}$ if and only if
\begin{equation}\label{eq:main4}
\int\limits_{\mathbb{C}_\frac{1}{2}}\cdots\int\limits_{\mathbb{C}_\frac{1}{2}}
\zeta''\left(\overline{w_1}+w_2\right)\cdots\zeta''\left(\overline{w_{m-1}}+w_m\right)\zeta''\left(\overline{w_m}+w_1\right)\prod_{k=1}^{m}M_\varphi(w_k)\,dA(w_k)<\infty.
\end{equation}
\end{theorem}

Our next result is a comparison-type principle. Using the Lindel\"{o}f principle for Green's functions, we will be able to establish geometric conditions on the symbols that 
imply that the associated composition operator is compact or belongs to $S_p.$
To our knowledge this is the first example of a technique that gives geometric conditions that apply to all symbols $\varphi\in\mathfrak{G}_0$. To exemplify this, we focus 
on symbols whose range is contained in angular sectors.

\begin{theorem}\label{thm:angularsectorzero}
 Let $\varphi\in\GHZ$ and assume that $\varphi(\CC_0)\subset\left\{s\in\Cud:\ \left|\arg(s)-\frac 12\right|<\frac{\pi}{2\alpha}\right\}$ for some $\alpha>1.$
 Then $C_\varphi$ is compact. If we further assume that $\alpha\leq 2,$ then $C_\varphi\in S_{2p}$ for any $p>1/(\alpha-1)$.
\end{theorem}

We can strengthen the previous result proving that if the range of the symbol meets the boundary inside a finite union of angular sectors, then the induced composition operator is compact. 

This geometric method also applies to continuity and compactness of composition operators acting on the other Hardy space of Dirichlet series $\mathcal H^p,$ $p\neq 2.$
Recall that for $0< p < \infty$, the Hardy space $\mathcal{H}^p$ of Dirichlet series is defined as the completion of Dirichlet polynomials under the Besicovitch norm (or quasi-norm if $0<p<1$)
\begin{equation*}
\norm{P}_{\mathcal{H}^p}:= \left(\lim\limits_{T\rightarrow\infty} \frac{1}{2T} \int\limits_{-T}^{T} |P(it)|^p\,dt\right)^{\frac{1}{p}}.
\end{equation*}

The characterization of bounded composition operators with Dirichlet series symbols on $\mathcal{H}^p,\,p\notin 2\mathbb N$ is an open and challenging question.
The condition $\varphi\in\mathfrak{G}_0$ is necessary but not sufficient, \cite{QQ20} and there is no known sufficient conditions which may be applied
to a large class of symbols whose range touches the boundary of $\CC_0.$ We provide such a sufficient condition under the assumption that the range of the symbol is 
contained in an angular sector.
\begin{theorem}\label{thm:p<2}
Let $k\in\mathbb N$ and $p\in (0,2k]$.
If the symbol $\varphi\in\mathfrak{G}_0$ maps the right half-plane into an angular sector of the form 
$\Omega=\left\{s\in\Cud:\ |\arg\left(s-\frac{1}{2}\right)|<\frac{p\pi}{4k}\right\}$, then $C_\varphi$ is bounded on $\mathcal{H}^p$. Furthermore, if $q\in (\max(1,\,p),\,2k]$, then the composition operator is compact on $\mathcal{H}^q$.
\end{theorem}

In the last section we briefly discuss the case of Bergman spaces of Dirichlet series as well as some results on Carleson measures.

\subsection*{Notation}
Throughout the article, we will be using the convention that $C$ denotes a positive constant which may vary from line to line.
We will write that $C = C(x)$ to indicate that the constant depends on a parameter $x$.
If $f,g$ are two real functions defined on the same set $\Omega,$ we will write $f\ll g$ if there exists $C>0$ 
such that for all $x\in\Omega,$ $f(x)\leq C g(x)$ and $f\sim g$ if $f\ll g$ and $g\ll f$.

\subsection*{Acknowledgments}  
We thank Ole Fredrik Brevig and Karl--Mikael Perfekt for providing helpful comments. 

Part of the work has been conducted during research visits of the second author at Université Clermont Auvergne and Lund University (supported from the project Pure Mathematics in Norway – Ren matematikk i Norge funded by the Trond Mohn Foundation). He wants to express his gratitude toward Alexandru Aleman and the previously mentioned departments for their hospitality.

\section{Background material}
\subsection{Schatten classes}
A compact operator $T$ acting on a separable Hilbert space $H$ can be written as
\begin{equation}\label{eq:canonicaldecomposition}
T(x)=\sum\limits_{n\geq 1}s_n\langle x, e_n\rangle h_n,\qquad x\in H,
\end{equation}
where $\{s_n\}_{n\geq 1}$ is the sequence of singular values and $\{e_n\}_{n\geq 1}$ and $\{h_n\}_{n\geq 1}$ are orthonormal sequences. 
In case $T$ is self-adjoint, then $e_n=h_n$ for all $n\geq 1.$ 

For $p>0$ the $S_p$ Schatten class of compact operators $T$ on $H$ is defined as
\begin{equation*}\label{defSchat}
S_p=S_p(H)=\left\{T\in\mathfrak K(H):\,\norm{T}_{S_p}^p:=\sum\limits_{n\geq 1} s_n^p<\infty\right\}.
\end{equation*}
Equivalently (see \cite{BLZ13}), for $p\geq 1,$ a bounded linear operator $T\in\mathfrak L(H)$ belongs to $S_p$ if and only if there exists a positive constant $C$ such that 
$$\sum_n |\langle Te_n,e_n\rangle|^p\leq C$$
for every orthonormal basis $(e_n)$. Furthermore, if $T$ is self-adjoint,
$$\norm{T}_{S_p}^p=\sup \sum_n |\langle Te_n,e_n\rangle|^p$$
the supremum being taken over all orthonormal basis of $H.$


For a positive operator $T$ on $H$ we define the power $T^p,\,p>0$, as
\begin{equation*}\label{eq:power}
T^p(x)=\sum\limits_{n\geq 1}s_n^p\langle x, e_n\rangle e_n,\qquad x\in H.
\end{equation*}
When $p=n\in\mathbb{N}$, the operator $T^n$ is the $n$-th iteration of $T$.
We observe that $T\in S_p$ if and only if $T^p\in S_1$. 
It $T$ is not assumed to be positive, we can still use that $T\in S_p$ iff $|T|^p=(T^*T)^{p/2}\in S_1$ iff $T^*T\in S_{p/2}$.

For a unit vector $x\in H$ and a positive operator $T$, applying  H\"{o}lder's inequality in \eqref{eq:canonicaldecomposition} we obtain the following inequality
\begin{equation}\label{eq: Holder}
\langle T^p(x),x\rangle\geq \left(\langle T(x),x\rangle\right)^p,\qquad p\geq 1.
\end{equation}
For $0<p\leq 1$ the inequality is reversed.

	\subsection{The infinite polytorus and vertical limits}
The infinite polytorus $\TT^\infty$ is defined as the (countable) infinite Cartesian product of copies of the unit  circle $\mathbb{T}$,
$$\mathbb{T}^\infty=\left\{\chi=(\chi_1,\chi_2,\dots): \,\chi_j\in\mathbb{T},\, j\geq 1\right\}.$$
It is a compact abelian group with respect to coordinatewise multiplication. We can identify the Haar measure $m_\infty$ of the infinite polytorus with the countable infinite product measure $m\times m\times\cdots$, where $m$ is the normalized Lebesgue measure of the unit circle.

$\mathbb{T}^\infty$ is isomorphic to the group of characters of $(\mathbb{Q}_+,\cdot)$. Given a point $\chi=(\chi_1,\chi_2,\dots)\in\mathbb{T}^\infty$, the coresponding character $\chi:\mathbb{Q}_+\rightarrow\mathbb{T}$ is the completely multiplicative function on $\mathbb{N}$ such that $\chi(p_j)=\chi_j$, where $\{p_j\}_{j\geq1}$ is the increasing sequence of primes, extended to $\mathbb{Q}_+$ through the relation $\chi(n^{-1})=\overline{\chi(n)}$.

Suppose $f(s)=\sum\limits_{n\geq1}\frac{a_n}{n^s}$ is a Dirichlet series and $\chi(n)$ is a character. The vertical limit function $f_\chi$ is defined as
$$f_\chi(s)=\sum\limits_{n\geq1}\frac{a_n\chi(n)}{n^s}.$$
By Kronecker's theorem \cite{BOH34}, for any $\epsilon>0$, there exists a sequence of real numbers $\{t_j\}_{j\geq1}$ such that $f(s+i t_j)\rightarrow f_\chi(s)$ uniformly on $\mathbb{C}_{\sigma_u(f)+\epsilon}$, where $\sigma_u(f)$ denotes the abscissa of uniform convergence of $f.$

If $f\in \mathcal{H}^2$, then for almost every character $\chi\in \mathbb{T}^\infty$ the vertical limit function $f_\chi$ converges in the right half-plane and has boundary values $f_{\chi}(it)=\lim\limits_{\sigma\rightarrow0^+}f_\chi(\sigma+it)$ for almost every $t\in\mathbb{R}$, \cite{BAY02}. For $\psi(s)=c_0s+\varphi(s)\in\mathfrak{G}$, we set
$$\psi_\chi(s)=c_0s+\varphi_\chi(s),$$
then for every $\chi\in\mathbb{T}^\infty$ we have that
\begin{equation} \label{eq:comprule}
\left(C_\psi(f)\right)_\chi=f_{\chi^{c_0}}\circ\psi_\chi.
\end{equation}
The symbol $\psi$ has boundary values $\psi_{\chi}(it)=\lim\limits_{\sigma\rightarrow0^+}\psi_\chi(\sigma+it)$ for almost every $\chi\in\mathbb{T}^\infty$  and for almost every $t\in\mathbb{R}$.

\subsection{Composition operators on $\mathcal{H}^2$}
O. F. Brevig and K--M. Perfekt \cite[Theorem~1.3.]{BP21} proved the following analogue of Stanton's formula for the Hardy spaces of Dirichlet series:
\begin{equation}\label{eq: Stanton}
\norm{ C_\varphi(f)}_{\mathcal{H}^2}^2=\left|f(\varphi(+\infty))\right|^2+\frac{2}{\pi}\int\limits_{\mathbb{C}_{\frac{1}{2}}}\left|f'(w)\right|^2M_{\varphi}(w) \, dA(w),
\end{equation}
where $\varphi\in\mathfrak{G}_0$ and $f\in\mathcal H^2$. By $f(+\infty)$ we denote the first coefficient $a_1$ of the Dirichlet series  $f(s)=\sum_{n\geq1}\frac{a_n}{n^s}$. We apply the polarization identity in \eqref{eq: Stanton} yielding to
\begin{equation}\label{pol}
\langle C_\varphi(f),C_\varphi(g)\rangle=f(\varphi(+\infty))\overline{g(\varphi(+\infty))}+\frac{2}{\pi}\int\limits_{\mathbb{C}_{\frac{1}{2}}}f'(w)\overline{g'(w)}M_{\varphi}(w) \, dA(w).
\end{equation}

We will make use of two properties of the counting function $M_{\varphi}(w)$ proved in \cite{BP21}, the submean value property and a Littlewood type inequality. Those respectively are
\begin{equation}\label{eq:submean}
M_{\varphi}(w)\leq \frac{1}{|D(w,r)|}\int\limits_{D(w,r)}M_{\varphi}(z) \, dA(z),
\end{equation}
for every disk $D(w,r)\subset \Cud$ that does not contain $\varphi(+\infty)$, and

	\begin{equation}\label{eq:littlewood}
	M_{\varphi}(w)\leq \log\left|\frac{\varphi(+\infty)+\overline{w}-1}{\varphi(+\infty)-w}\right|,\quad\quad w\in\Cud\backslash\{\varphi(+\infty)\}.
	\end{equation}

In Subsection~\ref{Green} we will prove a weaker version of the Littlewood inequality \eqref{eq:littlewood} but sufficient for our purpose.  
The standard technique to prove such inequalities goes through regularity results for conformal maps \cite{BP21, KP22}.
We shall use the following consequence of \eqref{eq:littlewood} (see \cite[Lemma 2.3]{BP21}): for $\sigma_\infty>\Re (\varphi(+\infty))$, 
there exists $C>0$ such that, for all $w\in\mathbb C_{\sigma_\infty},$
\begin{equation}\label{eq:superlittlewood}
M_\varphi(w)\leq C \frac{\Rewud}{\big(1+|\Im(w)|\big)^2}.
\end{equation}

\subsection{Carleson measures}
Let $H$ be a Hilbert space of holomorphic functions on a domain $\Omega$. A Borel measure $\mu$ in $\Omega$
is called a Carleson measure for $H$ if there exists a constant $C>0$ such that, for all $f\in H,$
$$
\int_{\Omega} |f(w)|^2d\mu(w)\leq C\norm{f}_H^2.
$$
We will denote by $C(\mu,H)$ or simply by $C(\mu)$ the infimum of such constants.
For instance, Carleson measures on the Hardy space $H^2(\mathbb{C}_{\frac{1}{2}})$, that consist of holomorphic function $f$ in $\mathbb{C}_{\frac{1}{2}}$ equipped with norm
$$\norm{f}_{H^2(\mathbb{C}_{\frac{1}{2}})} ^2:=\sup\limits_{\sigma>\frac{1}{2}}\int\limits_{\mathbb{R}}\left|f(\sigma+it)\right|^2dt<\infty,$$
are characterized as follows.
\begin{theorem}[\cite{C62}]
	A Borel measure $\mu$ on $\mathbb{C}_{\frac{1}{2}}$ is a Carleson measure for $H^2(\mathbb{C}_{\frac{1}{2}})$ if and only if there exists a constant $D>0$ such that for every square $Q$ with one side $I$ on the line $\{\Re s=\frac{1}{2}\}$
	\begin{equation}\label{carleson}
	\mu(Q)\leq D \left|I\right|.
	\end{equation}
Moreover, there exist two absolute constants $a,b>0$ such that, for all Borel measures $\mu$ on $\Cud,$ denoting by $D(\mu)$
the infimum of the constants $D$ verifying \eqref{carleson}, then $aD(\mu)\leq C(\mu)\leq bD(\mu)$.	
\end{theorem}

\subsection{Weighted Hilbert spaces of Dirichlet series}
Our main strategy (inspired by \cite{LZ92}) to obtain the membership of $C_\varphi$ to $S_{2p}$
is to derive it from the membership to $S_p$ of an associated Toeplitz operator defined on another space of Dirichlet series.
We now introduce this class of spaces.
For $a\leq 1$ we define the weighted Hilbert space $\mathcal{D}_a$ of Dirichlet series as 
$$\mathcal{D}_a=\left\{f(s)=\sum_{n\geq1}\frac{a_n}{n^s}:\norm{f}_a^2=|a_1|^2+\sum_{n\geq2}|a_n|^2\log(n)^a<\infty \right\}.$$
The reproducing kernel $k_{w,-a},\,a\geq 0$ of $\left(\mathcal{D}_{-a}\right)_0$ (space mod constants) at a point $w\in\mathbb{C}_{\frac{1}{2}}$ is given by 

\begin{equation}\label{rep}
k_{w,-a}(s)=\sum\limits_{n>1}\frac{(\log n)^a}{n^{s+\overline{w}}}=\frac{\Gamma(1+a)}{\left(\overline{w}+s-1\right)^{1+a}}+ E_a(s+\overline{w}), \qquad s\in\mathbb{C}_\frac{1}{2},
\end{equation}
where $E_a(\cdot)$ is a holomorphic function on $\mathbb{C}_{0}$, \cite[Lemma~5.1]{KP22}. Observe that
$$\|k_{w,-a}\|_{-a}^2\sim_{\Re(w)\to\frac 12}\frac1{\PRewud^{a+1}}.$$

The local embedding theorem, \cite{HLS97}, states that there exists an absolute constant $C>0$ such that for every $f\in\mathcal{H}^2$
\begin{equation}\label{eq:localembending}
\frac{1}{T}\int\limits_{-T}^T\left|f\left(\frac{1}{2}+it\right)\right|^2\,dt\leq C \norm{f}^2_{\mathcal{H}^2},\qquad T>0.
\end{equation}

A direct application of \eqref{eq:localembending} is that for every $f\in\left(\mathcal{D}_{-a}\right)_0,\,a>0,$ we have the following embedding
\begin{equation}\label{eq:localbergmanemben}
\frac{1}{T}\int\limits_{-T}^T\int\limits_{\frac{1}{2}}^\infty\left|f(\sigma+it)\right|^2\left(\sigma-\frac{1}{2}\right)^{a-1}\,d\sigma\,dt\leq C(a) \norm{f}^2_{-a}.
\end{equation}
In particular, if $A$ is a subset of $\Cud$ with bounded imaginary part, then $\mathbf 1_A\left(\Re (\cdot)-\frac12\right)^{a-1} dA$
is a Carleson measure for $(\mathcal D_{-a})_0.$ More generally, if $\kappa:[0,+\infty)\to[0,+\infty)$ is integrable, bounded and decreasing then $\kappa(\left|\Im(\cdot)\right|)\left(\Re (\cdot)-\frac12)\right)^{a-1}dA$
is a Carleson measure for $(\mathcal D_{-a})_0.$

The differentiation operator $D(f)=f'$ is an isometry between $\mathcal{H}^2_0$ and $\left(\mathcal{D}_{-2}\right)_0$.
By \eqref{pol} the composition operator $C_\varphi$ belongs to $S_{2p}(\mathcal H^2),\,p>0$ if and only if 
$\left(D\circ C_\varphi\circ D^{-1}\right)^*D\circ C_\varphi\circ D^{-1}\in S_{p}((\mathcal D_{-2})_0)$
if and only if the operator $T_\varphi:\left(\mathcal{D}_{-2}\right)_0\rightarrow \left(\mathcal{D}_{-2}\right)_0$ defined as
\begin{equation}\label{tf}
\langle T_\varphi(f),g\rangle=\int\limits_{\mathbb{C}_{\frac{1}{2}}}f(w)\overline{g(w)}M_\varphi(w)\,dA(w)
\end{equation}
belongs to $S_{p}((\mathcal D_{-2})_0)$.


\section{Composition operators belonging to Schatten classes}
\subsection{Schatten class and Carleson measures}

We shall divide the proof of Theorem \ref{boundedchar} into several parts. We first handle the case $p>1$ in a more general context
by giving a necessary and a sufficient condition for $C_\varphi$ to belong to $S_{2p}$. Both conditions involve $M_\varphi$
and Carleson measures. At this stage, we do not assume anything on the imaginary part of $\varphi.$

\begin{theorem}\label{thm:schattencarleson}
 Let $p>1$ and $\varphi\in\GHZ$. 
 \begin{enumerate}[a)]
  \item Assume that $C_\varphi\in S_{2p}$ and let $\mu$ be a Carleson measure for $(\mathcal D_{-2})_0.$ Then
  $$\int_{\Cud} \frac{(M_\varphi(w))^p \zeta''(2\Re(w))}{\PRewud^p} d\mu(w)<+\infty.$$
  \item Assume that there exists $\rho:\varphi(\mathbb C_0)\to(0,+\infty)$ such that $\rho dA$ is a Carleson measure for $(\mathcal D_{-2})_0$
  and that
\begin{equation}\label{eq:sufficientcarleson}
  \int_{\varphi(\mathbb C_0)}\frac{(M_\varphi(w))^p\zeta''(2\Re(w))}{\rho(w)^{p-1}}dA(w)<+\infty.
\end{equation}
  Then $C_\varphi\in S_{2p}.$
 \end{enumerate}
\end{theorem}

\begin{proof}
 We start by proving a). Let $I_\mu$ be the inclusion operator from $(\mathcal D_{-2})_0$ into $L^2(\Cud,\mu)$ which is bounded since $\mu$ is Carleson. Moreover, assuming $C_\varphi\in S_{2p}$ or, equivalently,
$T_\varphi\in S_p$, we get by the ideal property of Schatten classes that $I_\mu\circ T_\varphi^{p/2}\in S_2.$ Let $(f_n)$ be any orthonormal sequence of $(\mathcal D_{-2})_0.$
One can write
\begin{align*}
\infty>\|I_\mu\circ T_\varphi^{p/2}\|_{S_2}^2&=\sum_{n\geq 1}\|T_\varphi^{p/2}(f_n)\|_{L^2(\mu)}^2 \\
&=\sum\limits_{n\geq 1}\int\limits_{\mathbb{C}_\frac{1}{2}}\left|\langle T_\varphi^{\frac{p}{2}}\left(f_n\right),k_{w,-2}\rangle\right|^2\,d\mu(w)\\
&=\int\limits_{\mathbb{C}_\frac{1}{2}} \norm{T_\varphi^{\frac{p}{2}}\left(k_{w,-2}\right)}^2_{(\mathcal D_{-2})_0}\,d\mu(w)\\
&=\int\limits_{\mathbb{C}_\frac{1}{2}} \langle T_\varphi^p\left(k_{w,-2}\right),k_{w,-2}\rangle \,d\mu(w)\\
&\geq \int\limits_{\mathbb{C}_\frac{1}{2}} \left(\langle T_\varphi\left(K_{w,-2}\right),K_{w,-2}\rangle\right)^p\norm{k_{w,-2}}^2 \,d\mu(w)
\end{align*}
by \eqref{eq: Holder}, where $K_{w,-2}$ is the normalized reproducing kernel of $(\mathcal D_{-2})_0$ at $w$. Observe that the 
exchange of integral and sum is justified by Tonelli's theorem.  Fix $\sigma_\infty>\Re \varphi(+\infty)$. By \eqref{tf},
\begin{align*}
\|I_\mu\circ T_\varphi^{p/2}\|_{S_2}^2&\geq \int_{\Cud}\left(\int_{\Cud}|K_{w,-2}(z)|^2M_\varphi(z)dA(z)\right)^p\norm{k_{w,-2}}^2 \,d\mu(w)\\
&\geq \int_{\Cud\setminus\mathbb{C}_{\sigma_\infty}}\left(\int_{D\left(w,\frac 12\PRewud\right)}|K_{w,-2}(z)|^2M_\varphi(z)dA(z)\right)^p\norm{k_{w,-2}}^2 \,d\mu(w).
\end{align*}
By \eqref{rep} one can estimate the behaviour of $K_{w,-2}(z)$ in the disc $D\left(w,\frac{\Re w-\frac{1}{2}}{2}\right),$ for $\Re w\leq \sigma_\infty$ to obtain
\begin{align*}
\int_{D\left(w,\frac 12\PRewud\right)}|K_{w,-2}(z)|^2M_\varphi(z)dA(z)&\gg \int_{D\left(w,\frac 12\PRewud\right)}
\frac{M_\varphi(z)}{\PRewud^3 }dA(z)\\
&\gg \frac{M_\varphi(w)}{\Rewud}
\end{align*}
where the last inequality follows from the submean value property of the mean counting function \eqref{eq:submean}. Taking into account the value of $\norm{k_{w,-2}}$ we get 
 $$\int_{\Cud\setminus\mathbb{C}_{\sigma_\infty}} \frac{(M_\varphi(w))^p \zeta''(2\Re(w))}{\PRewud^p} d\mu(w)<+\infty.$$
Finally, \eqref{eq:superlittlewood} yields 
$$\int_{\mathbb{C}_{\sigma_\infty}} \frac{(M_\varphi(w))^p \zeta''(2\Re(w))}{\PRewud^p} d\mu(w)\ll \int_{\mathbb{C}_{\sigma_\infty}} \zeta''(2\Re(w))d\mu(w)\ll\int_{\mathbb{C}_{\sigma_\infty}}\left|2^{-w}\right|^2d\mu(w) <+\infty.$$
Conversely, assume that \eqref{eq:sufficientcarleson} holds and let $q$ be the conjugate exponent of $p$. 
Let $(f_n)$ be any orthonormal basis of $(\mathcal D_{-2})_0.$ Then
\begin{align*}
\sum_{n\geq 1}\langle T_\varphi(f_n),f_n\rangle^p&=\sum_{n\geq 1}\left(\int_{\Cud}|f_n(w)|^2 M_\varphi(w) dA(w)\right)^p\\
&=\sum_{n\geq 1}\left(\int_{\varphi(\mathbb C_0)}\frac{|f_n(w)|^{2/p} M_\varphi(w)}{\rho(w)^{1/q}} |f_n(w)|^{2/q}\rho(w)^{1/q}dA(w)\right)^p\\
&\leq \sum_{n\geq 1}\left(\int_{\varphi(\mathbb C_0)}\frac{|f_n(w)|^{2} M_\varphi(w)^p}{\rho(w)^{p/q}}dA(w)\right)
\left(\int_{\varphi(\mathbb C_0)} |f_n(w)|^2 \rho(w)dA(w)\right)^{p/q}.
\end{align*}
Since $\rho dA$ is a Carleson measure and since $\sum_n |f_n(w)|^2=k_{w,-2}(w)$ for any orthonormal basis of $(\mathcal D_{-2})_0,$
we get
\begin{align*}
\sum_{n\geq 1}\langle T_\varphi(f_n),f_n\rangle^p&\ll \int_{\varphi(\mathbb C_0)}\frac{(M_\varphi(w))^p k_{w,-2}(w)}{\rho(w)^{p-1}}dA(w)\\
&\ll \int_{\varphi(\mathbb C_0)}\frac{(M_\varphi(w))^p\zeta''(2\Re(w))}{\rho(w)^{p-1}}dA(w).
\end{align*}
Hence, $T_\varphi$ belongs to $S_p$.
\end{proof}

In view of the above theorem, the ideal case would be to choose a function $\rho:\varphi(\mathbb C_0)\to (0,+\infty)$ such that
$\rho dA$ is a Carleson measure for $(\mathcal D_{-2})_0$ and 
$$\frac1{\rho(w)^{p-1}}=\frac{\rho(w)}{\PRewud^p},\ w\in\varphi(\mathbb C_0).$$
This yields to $\rho(w)=\Rewud$. Now if $\varphi$ has bounded imaginary part, then the embedding inequality
implies that $\mathbf 1_{\varphi(\mathbb C_0)}\PRewud dA$ is a Carleson measure for $(\mathcal D_{-2})_0.$ 
This gives the way to the case $p>1$ of Theorem \ref{boundedchar}.
\begin{corollary}\label{cor:schatten}
 Let $p>1$ and $\varphi\in \GHZ$ with bounded imaginary part. Then 
 $C_\varphi$ belongs to $S_{2p}$ if and only if 
\begin{equation*}
 \int\limits_{\mathbb{C}_\frac{1}{2}}\frac{\left(M_\varphi(w)\right)^{p}}{\left(\Re w-\frac{1}{2}\right)^{p+2}}\, dA(w)<+\infty.
\end{equation*}
\end{corollary}
\begin{proof}
 Our discussion actually shows that, under the assumptions of the corollary, $C_\varphi\in S_{2p}$ if and only if
 $$\int\limits_{\mathbb{C}_\frac{1}{2}}\frac{\left(M_\varphi(w)\right)^{p}}{\left(\Re w-\frac{1}{2}\right)^{p-1}}\zeta''(2\Re(w))\, dA(w)<+\infty.$$
 It remains to show that this is equivalent to \eqref{eq:main1}. Let $\sigma_\infty=2\Re \varphi(+\infty).$
 Then for $w\in\Cud\backslash \mathbb C_{\sigma_\infty}, $
 $$\frac{1}{\PRewud^3}\ll\zeta''(2\Re(w))\ll \frac{1}{\PRewud^3}.$$
We may conclude if we are able  to prove that for any $\varphi\in\GHZ$,
$$\int_{\mathbb C_{\sigma_\infty}}\frac{M_\varphi(w)^p}{\PRewud^{p-1}}\zeta''(2\Re(w))dA(w)<+\infty
\textrm{ and }
\int_{\mathbb C_{\sigma_\infty}}\frac{M_\varphi(w)^p}{\PRewud^{p+2}}dA(w)<+\infty.$$
Both of these properties follow from \eqref{eq:superlittlewood}.
\end{proof}

When $\varphi$ does not have bounded imaginary part, there are still interesting Carleson measures for $(\mathcal D_{-2})_0,$ for instance
$\PRewud/(1+|\Im(w)|)^adA$
for any $a>1.$ This yields to the following result.

\begin{corollary}\label{cor:carleson}
 Let $p>1$, let $\varphi\in \GHZ$ and let $a>1.$ 
 \begin{enumerate}[a)]
  \item If $C_\varphi$ belongs to $S_{2p},$ then
  $$\int_{\Cud}\frac{M_\varphi(w)^p}{\PRewud^{p+2}(1+|\Im(w)|)^a}dA(w)<+\infty.$$
  \item Assume that 
  $$\int_{\Cud}\frac{M_\varphi(w)^p}{\PRewud^{p+2}}(1+|\Im(w)|)^{a(p-1)}dA(w)<+\infty.$$
  Then $C_\varphi\in S_{2p}.$
 \end{enumerate}
\end{corollary}
 \begin{proof}
  This follows from Theorem \ref{thm:schattencarleson} with $\rho(w)=\PRewud/(1+|\Im(w)|)^a$ and $d\mu=\rho dA.$
  Again we can replace everywhere $\zeta''(2\Re(w))$ by $\PRewud^{-3}$ since for a),
  $$\int_{\mathbb C_{\sigma_\infty}} \frac{M_\varphi(w)^p}{\PRewud^{p+2}(1+|\Im(w)|)^a}dA(w)<+\infty$$
  and for b), $\zeta''(2\Re(w))\ll \PRewud^{-3}$ is valid throughout $\Cud.$
 \end{proof}
 
We now prove that \eqref{eq:main1} remains necessary for $p\geq 2$ without any assumption on $\varphi.$
\begin{theorem}\label{thm:necessaryp>2}
 Let $p\geq 2$ and $\varphi\in\GHZ.$ Assume that $C_\varphi\in S_{2p}.$ Then
$$ \int\limits_{\mathbb{C}_\frac{1}{2}}\frac{\left(M_\varphi(w)\right)^{p}}{\left(\Re w-\frac{1}{2}\right)^{p+2}}\, dA(w)<+\infty.$$
\end{theorem}
\begin{proof}
 For the positive operator $T_\varphi$ belonging to $S_p,$ denoting by $(f_n)$ an orthonormal basis of eigenvectors of $T_\varphi,$
 \begin{align*}
\infty>\norm{T^p_\varphi}_{S_1}&=\sum\limits_{n\geq 1}\langle T^p_\varphi(f_n),f_n\rangle\\
&=\sum\limits_{n\geq 1}\int\limits_{\mathbb{C}_{\frac{1}{2}}}T^{p-1}_\varphi(f_n)(w)\overline{f_n(w)}M_\varphi(w)\,dA(w)\\
&=\sum\limits_{n\geq 1}\int\limits_{\mathbb{C}_{\frac{1}{2}}}\langle T^{p-1}_\varphi(f_n), k_{w,-2}\rangle\overline{f_n(w)}M_\varphi(w)\,dA(w).
\end{align*}
The quantity under the integral sign is nonnegative since
$$\langle T_\varphi^{p-1}(f_n),k_{w,-2}\rangle\overline{f_n(w)}=s_n^{p-1}\langle f_n,k_{w,-2}\rangle\overline{f_n(w)}=s_n^{p-1}|f_n(w)|^2.$$
An application of Tonelli's theorem yields
\begin{align*}
\infty>\norm{T^p_\varphi}_{S_1}
&=\int\limits_{\mathbb{C}_{\frac{1}{2}}}\sum\limits_{n\geq 1}\langle f_n, T^{p-1}_\varphi(k_{w,-2})\rangle\overline{f_n(w)}M_\varphi(w)\,dA(w)\\
&=\int\limits_{\mathbb{C}_{\frac{1}{2}}}\langle T^{p-1}_\varphi(k_{w,-2}),k_{w,-2} \rangle M_\varphi(w)\,dA(w)\\
&=\int\limits_{\mathbb{C}_{\frac{1}{2}}}\langle T^{p-1}_\varphi(K_{w,-2}),K_{w,-2} \rangle \norm{k_{w,-2}}^2 M_\varphi(w)\,dA(w).
\end{align*}
We now use  \eqref{eq: Holder}
\begin{align*}
\infty>\norm{T^p_\varphi}_{S_1}&\geq \int\limits_{\mathbb{C}_{\frac{1}{2}}}\langle T_\varphi(K_{w,-2}),K_{w,-2} \rangle^{p-1} \norm{k_{w,-2}}^2 
M_\varphi(w)\,dA(w)\\
&\geq \int\limits_{\mathbb{C}_{\frac{1}{2}}}\left(\int\limits_{\mathbb{C}_{\frac{1}{2}}}\left|K_{w,-2}(z)\right|^2M_\varphi(z)\,dA(z)\right)^{p-1} \norm{k_{w,-2}}^2 M_\varphi(w)\,dA(w).
\end{align*}
We conclude as above using the submean value property of the counting function \eqref{eq:submean} to deduce that \eqref{eq:main1} holds true.
\end{proof}

We end up the proof of Theorem \ref{boundedchar} by considering the case $p\in(0,1)$.
\begin{theorem}\label{thm:pinzeroone}
 Let $p\in (0,1)$ and $\varphi\in \GHZ.$ Assume that $\varphi$ has bounded imaginary part and $C_\varphi\in S_{2p}$. Then
 $$\int\limits_{\mathbb{C}_\frac{1}{2}}\frac{\left(M_\varphi(w)\right)^{p}}{\left(\Re w-\frac{1}{2}\right)^{p+2}}\, dA(w)<+\infty.$$
\end{theorem}
\begin{proof}
 We still denote by $(f_n)$ an orthonormal basis of eigenvectors of $T_\varphi.$ We now write
 \begin{align*}
  \|T_\varphi\|_{S_p}^p&=\sum_{n\geq 1}\left(\langle T_\varphi (f_n),f_n\rangle\right)^p \\
  &=\sum_{n\geq 1}\left(\int_{\Cud}|f_n(w)|^2 M_\varphi(w) dA(w)\right)^p \\
  &=\sum_{n\geq 1}\left(\int_{\varphi(\mathbb C_0)}\frac{M_\varphi(w)}{\Rewud} |f_n(w)|^2 \PRewud dA(w)\right)^p.
 \end{align*}
 Now by \eqref{eq:localbergmanemben} the measures $\mathbf 1_{\varphi(\mathbb C_0)}|f_n(\cdot)|^2 \left(\Re (\cdot)-\frac 12\right) dA$ are finite measures on $\Cud$
 with uniformly bounded mass. It follows from H\"{o}lder inequality that 
 \begin{align*}
  \|T_\varphi\|_{S_p}&\gg \int_{\Cud}\sum_{n\geq 1} \frac{M_\varphi(w)^p}{\PRewud^p} |f_n(w)|^2 \PRewud dA(w)\\
  &\gg \int_{\Cud} \frac{M_\varphi(w)^p}{\PRewud^{p-1}}\zeta''(2\Re(w))dA(w).
 \end{align*}
\end{proof}

\subsection{The case of even integers}

We now prove the $2m$-Schatten class characterization \eqref{eq:main4}.

\begin{proof}[\textbf{Proof of Theorem \ref{main5}}]
We first prove that \eqref{eq:main4} implies that $C_\varphi$ is compact.
If this was not the case, then we could find $\delta>0$ and a sequence $(w(k))\subset\Cud$ with real part going to $1/2$ such that for every $\varepsilon\in(0,1)$ the rectangles
\begin{multline*}
R_k=\\\left(\frac{\Re w(k)-\frac{1}{2}}{2},\frac{3\left(\Re w(k)-\frac{1}{2}\right)}{2}\right)\times \left(\Im w(k)-\varepsilon\left(\Re w(k)-\frac{1}{2}\right),\Im w(k)+\varepsilon\left(\Re w(k)-\frac{1}{2}\right)\right)
\end{multline*} 
are pairwise disjoint
and for all $k\geq 1,$
$$\frac{M_\varphi(w(k))}{\Re(w(k))-\frac 12}\geq\delta.$$

Let $A_k=\prod_{j=1}^m R_k$. We recall that $\zeta''(s)$ has a pole of order $3$ at $s=1$, thus we can choose $\varepsilon>0$ close to zero such that 
$$\Re\left(\zeta''\left(\overline{w_1}+w_2\right)\cdots\zeta''\left(\overline{w_{m-1}}+w_m\right)\zeta''\left(\overline{w_m}+w_1\right)\right)\gg \left(\Re w(k)-\frac{1}{2}\right)^{-3m},$$
for every $w=(w_1,\dots,w_m)\in A_k$. Using the mean-value property of the counting function as well as the estimate above, we obtain that
$$\int_{A_k} \zeta''\left(\overline{w_1}+w_2\right)\cdots\zeta''\left(\overline{w_{m-1}}+w_m\right)\zeta''\left(\overline{w_m}+w_1\right)\prod_{j=1}^{m}M_\varphi(w_j)\,dA(w_j)\gg
\prod_{j=1}^m \frac{M_\varphi(w_k)}{\Re (w(k))-\frac 12}\geq\delta^m.$$
Since the sets $A_k$ are pairwise disjoint, this would contradict \eqref{eq:main4}.

Hence, for both implications of Theorem \ref{main5}, we may assume that $C_\varphi$ hence $T_\varphi$ is compact.
Let us consider the canonical decomposition of $T_\varphi$, $T_\varphi(f)=\sum\limits_{n\geq 1}s_n\langle f,f_n\rangle f_n.$
We know that
$C_\varphi\in S_{2m}$ if and only if $T^m_\varphi\in S_{1}$ if and only if
\begin{equation*}
\sum\limits_{n\geq 1}\langle T_\varphi^m(f_n),f_n\rangle<\infty.
\end{equation*}
We observe that 
\begin{equation*}
\sum\limits_{n\geq 1}\langle T_\varphi^m(f_n),f_n\rangle=\sum\limits_{n\geq 1}\int\limits_{\mathbb{C}_\frac{1}{2}}\langle T_\varphi^{m-1}(f_n),k_{w_1,-2}\rangle\overline{f_n(w_1)}M_\varphi(w_1)\,dA(w_1).
\end{equation*}
Arguing as in the proof of Theorem \ref{thm:necessaryp>2}, we may use Tonelli's theorem to get
\begin{align*}
\sum\limits_{n\geq 1}\langle T_\varphi^m(f_n),f_n\rangle&=\int\limits_{\mathbb{C}_\frac{1}{2}}\langle T_\varphi^{m-1}(k_{w_1,-2}),k_{w_1,-2}\rangle M_\varphi(w_1)\,dA(w_1)\\
&=\int\limits_{\mathbb{C}_\frac{1}{2}}\int\limits_{\mathbb{C}_\frac{1}{2}}\langle T_\varphi^{m-2}(k_{w_1,-2}),k_{w_2,-2}\rangle\zeta''(w_1+\overline{w_2}) M_\varphi(w_2)M_\varphi(w_1)\,dA(w_2)\,dA(w_1).
\end{align*}
By induction one obtains
\begin{equation*}
\|T_\varphi^m\|_{S_1}=\int\limits_{\mathbb{C}_\frac{1}{2}}\cdots\int\limits_{\mathbb{C}_\frac{1}{2}}\zeta''\left(\overline{w_1}+w_2\right)\cdots\zeta''\left(\overline{w_{m-1}}+w_m\right)\zeta''\left(\overline{w_m}+w_1\right)\prod_{k=1}^{m}M_\varphi(w_k)\,dA(w_k).
\end{equation*}
\end{proof}

We now intend to give a similar characterization involving the boundary values of $\varphi\in\GHZ.$ For every $\chi\in\mathbb T^\infty,$ $\varphi_\chi$ belongs to $\GHZ$ 
and for almost every $\chi,$ the generalized boundary value $\varphi(\chi)=\lim_{\sigma\to 0^+}\varphi_\chi(\sigma)$ does exist (see \cite[Section 2]{BP20} or \cite[Corollary 3.3]{BP21}).
Of course, $\Re(\varphi(\chi))\geq 1/2$ for almost every $\chi\in\mathbb T^\infty.$ 
We first show that when $C_\varphi$ is compact, this inequality is strict for almost
every $\chi\in\mathbb T^\infty.$

\begin{theorem}\label{thm:varphichi}
Let $\varphi\in\GHZ$ such that $C_\varphi$ induces a compact operator on $\mathcal H^2.$
Then $\Re(\varphi(\chi))>1/2$ for almost every $\chi\in\mathbb T^\infty.$
\end{theorem} 
\begin{proof}
The norm of the image of a function $f\in\mathcal{H}^2$ under $C_\varphi$ can be written as
\begin{equation*}
\norm{C_\varphi(f)}_{\mathcal{H}^2}^2=\int\limits_{\mathbb{T}^\infty}\left|f\circ\varphi(\chi)\right|^2\,dm_\infty(\chi)=\int\limits_{\overline{\mathbb{C}_\frac{1}{2}}}\left|f(w)\right|^2\,d\mu_\varphi(w),
\end{equation*}
where $\mu_\varphi$ if the push-forward measure of $m_\infty$ by $\varphi(\chi)$, see \cite{BP20}. Since $C_\varphi$ is compact and the reproducing kernel $\zeta(\cdot+\bar w)$ of $\mathcal H^2$ at $w$ 
satisfies
$$\zeta(s+\bar w)=\frac{1}{\bar w+s-1}+O(1),$$
we can argue like in the proof of \cite[Theorem 3]{OS12} to deduce that
\begin{equation}\label{Carleson}
\mu_\varphi\left(Q\right)=o\left(\left|I\right|\right),\qquad |I|\rightarrow 0,
\end{equation}
where $Q$ is a (Carleson) square in $\mathbb{C}_\frac{1}{2}$ with one side $I$ on the vertical line $\left\{\Re s=\frac{1}{2}\right\}$. This means that $\mu_\varphi$ is a vanishing Carleson measure for $H^2(\mathbb{C}_\frac{1}{2})$ and this implies that $\mu_\varphi\vline_{\left\{\Re s=\frac{1}{2}\right\}}$ is absolutely continuous with respect to the Lebesgue measure of $\mathbb{R}$. 
Following a standard argument, see for example \cite[Chapter~3]{CM98}, we will prove that $\mu_\varphi\vline_{\left\{\Re s=\frac{1}{2}\right\}}$ is equal to $0.$
 By the Lebesgue-Radon-Nikodym theorem there exists a positive function $f\in L^1(\mathbb{R})$ such that $$d\mu_\varphi\vline_{\left\{\Re s=\frac{1}{2}\right\}}=f(t)\,dt.$$
The set $E=\{\chi:\,\Re\varphi(\chi)>\frac{1}{2}\}$ is of full measure if and only if $f\equiv0$. Let us assume that there exists $\varepsilon>0$ such that $\left| f^{-1}\left((\varepsilon,\infty)\right)\right|>0.$
Let $F\subset f^{-1}((\varepsilon, +\infty))$ with positive and finite measure and let $\delta>0$ be such that 
$$\mu_{\varphi}(Q)\leq \frac{\veps}4 |I|$$
for every Carleson square in $\Cud$ with length $|I|\leq\delta.$ We can cover $F$ 
by a sequence of intervals $(I_n)$ such that $|I_n|\leq\delta$ and 
$$\sum_n |I_n|\leq 2|F|.$$
Now,
$$\veps |F|\leq \mu_\varphi\vline_{\left\{\Re s=\frac{1}{2}\right\}} (F)\leq \frac{\veps}4\sum_n |I_n|\leq \frac{\veps}2 |F|$$
which is a contradiction with $|F|>0$. Thus $\Re(\varphi(\chi))>1/2$ for a.e. $\chi\in\TT^\infty.$
\end{proof}

We are now ready to give an analogue of Theorem \ref{main5} involving the symbol directly.

\begin{theorem}\label{thm:schattenvarphichi}
Suppose that the symbol $\varphi\in \mathfrak{G}_0$ induces a compact operator and let $m\in\mathbb N.$ Then, $C_\varphi$ belongs to $S_{2m}$ if and only if
	\begin{equation}\label{eq:main5}
	\int\limits_{\mathbb{T}^\infty}\cdots\int\limits_{\mathbb{T}^\infty}\zeta\left(\overline{\varphi(\chi_1)}+\varphi(\chi_2)\right)\cdots\zeta\left(\overline{\varphi(\chi_{m-1})}+\varphi(\chi_m)\right)\zeta\left(\overline{\varphi(\chi_m)}+\varphi(\chi_1)\right)\prod_{k=1}^{m}\,dm_\infty(\chi_k)<\infty.
\end{equation}
\end{theorem}
\begin{proof}
Let $T=C^*_\varphi\circ C_\varphi$ and let us consider its canonical decomposition 
$$T(f)=\sum_{n\geq 1}s_n \langle f,f_n\rangle f_n.$$
We know that $C_\varphi\in S_{2m}\iff T^m \in S_1$ and that
$$\langle T(f),g\rangle=\int_{\TT^\infty}f(\varphi(\chi))\overline{g(\varphi(\chi))}dm_\infty(\chi).$$
Then
\begin{align*}
\sum_{n\geq 1}\langle T^m(f_n),f_n\rangle&=\sum_{n\geq 1}\langle T(f_n),T^{m-1}(f_n)\rangle\\
&=\sum_{n\geq 1}\int_{\TT^\infty}f_n(\varphi(\chi_1))\overline{\langle T^{m-1}(f_n),
\zeta(\cdot+\overline{\varphi(\chi_1)})\rangle} dm_\infty(\chi_1).
\end{align*}
As in the proof of Theorem \ref{main5} the quantity inside the integral is nonnegative which allows us to use Tonelli's theorem. Hence
\begin{align*}
\sum_{n\geq 1}\langle T^m(f_n),f_n\rangle&=\int_{\TT^\infty}\langle T^{m-1}(k_{\varphi(\chi_1),0}),k_{\varphi(\chi_1),0}\rangle dm_\infty(\chi_1)\\
&=\int_{\TT^\infty}\int_{\TT^\infty}\langle T^{m-2}(k_{\varphi(\chi_1),0}),k_{\varphi(\chi_2),0}\rangle \zeta(\varphi(\chi_1)+\overline{\varphi(\chi_2)})dm_\infty(\chi_1)dm_\infty(\chi_2).
\end{align*}
By induction one finally obtains
$$\|T^m\|_{S_1}=\int\limits_{\mathbb{T}^\infty}\cdots\int\limits_{\mathbb{T}^\infty}\zeta\left(\overline{\varphi(\chi_1)}+\varphi(\chi_2)\right)\cdots\zeta\left(\overline{\varphi(\chi_{m-1})}+\varphi(\chi_m)\right)\zeta\left(\overline{\varphi(\chi_m)}+\varphi(\chi_1)\right)\prod_{k=1}^{m}\,dm_\infty(\chi_k).$$
\end{proof}


\section{A comparison-type principle}\label{Green}
\subsection{Lindel\"{o}f principle and Littlewood inequality}
In this section we will use Lindel\"{o}f principle for Green's function to give a simple proof of a non-contractive Littlewood inequality \eqref{eq:littlewood}. Similar techniques have been used in the disk setting, \cite{B19}.

We recall (see for instance \cite{RAN95}) that a Green's function for a domain $\Omega\subset\CC$ is a function $g_\Omega:\Omega\times\Omega\to (-\infty,+\infty]$
such that, for all $w\in\Omega,$ $g(\cdot,w)$ is harmonic in $\Omega\backslash\{w\},$ $g_\Omega(z,w)\to 0$ n.e as $z\to\partial\Omega$ and $g_\Omega(\cdot,w)+\log|\cdot-w|$
is harmonic in a neighbourhood of $w$. If a domain admits a Green's function then it is necessarily unique. For instance, the Green's function on the disk
$g_\mathbb{D}:\mathbb{D}\times\mathbb{D}\mapsto(0,+\infty]$ has the form 
\begin{equation*}
g_\mathbb{D}(z,w)=\log\left|\frac{1-z\overline{w}}{z-w}\right|.
\end{equation*}
By conformal invariance we can easily define Green's function on every simply connected subdomain of the complex plane, for example
\begin{equation*}
g_{\mathbb{C}_0}(z,w)=\log\left|\frac{z+\overline{w}}{z-w}\right|,\qquad z,\,w\in\mathbb{C}_0.
\end{equation*}
The class of domains $D$ possessing a Green's function $g_{D}$ is much larger than the simply connected domains, see \cite[Chapter~4]{RAN95}. 
Lindel\"{o}f principle for Green's function (see for instance \cite{B14}) states that if $f$ is a holomorphic function mapping $D_1$ to $D_2$, where both of those domains possess Green's function, then for $z_0\in D_1$ and $w\in D_2\setminus\{f(z_0)\}$
\begin{equation}
\sum\limits_{z\in f^{-1}(\{w\})} g_{D_1}(z,z_0)\leq g_{D_2}(w,f(z_0)).
\end{equation}
Let us first show how to deduce, up to a multiplicative constant, Littlewood inequality \eqref{eq:littlewood} and also a corresponding inequality 
for a symbol in $\mathfrak{G}_{\geq1}$ (such an inequality was used in \cite{BAY21} to obtain a sufficient condition for composition operators with symbols in $\mathfrak{G}_{\geq1}$
to be compact on $\mathcal{H}^2$). Recall that for $\psi\in\mathfrak{G}_{\geq1}$, its restricted Nevanlinna counting function is defined by
$$N_\psi(w)=\sum_{\substack{s\in\psi_\chi^{-1}(\{w\})\\|\Im s|\leq 1}}\Re s.$$

\begin{theorem}\label{thm:littlewoodinequality}
 \begin{enumerate}[a)]
  \item Let $\varphi\in\GHZ.$ Then for all $w\in\Cud\backslash\{\varphi(+\infty)\},$
  $$M_\varphi(w)\leq\pi\log\left|\frac{\varphi(+\infty)+\overline w-1}{\varphi(+\infty)-w}\right|.$$
  \item Let $\psi\in\GHU.$ There exists $C>0$ such that, for all $\chi\in\TT^\infty,$ for all $w\in\CC_0$ with $\Re w\leq c_0,$
  $$N_{\psi_\chi}(w)\leq C\frac{\Re w}{1+(\Im w)^2}.$$
 \end{enumerate}
\end{theorem}
\begin{proof}
 Let $\varphi\in\GHZ$ and $w\in\Cud\backslash\{\varphi(+\infty)\}.$ For $T>0$ sufficiently large, $w\neq\varphi(T)$ so
 that Lindel\"{o}f principle implies
 $$\sum\limits_{s\in\varphi^{-1}(\{w\})}\log\left|\frac{T+\overline{s}}{T-s}\right|\leq \log\left|\frac{\varphi(T)+\overline{w}-1}{\varphi(T)-w}\right|.$$
 On the other hand, using the elementary inequality $\log(x)\geq \frac12(1-x^{-2})$ valid for $x>0,$
 $$ \sum\limits_{s\in\varphi^{-1}(\{w\})}\log\left|\frac{T+\overline{s}}{T-s}\right| \geq  \sum\limits_{s\in\varphi^{-1}(\{w\})}\frac{2T\Re s}{|T+s|^2}.$$
 Observe that $\{\Re(s):s\in\varphi^{-1}(\{w\})\}$ is bounded. Therefore, for all $\veps\in(0,1),$
 we can choose $T$ large enough so that
 \begin{align*}
  \sum\limits_{s\in\varphi^{-1}(\{w\})}\log\left|\frac{T+\overline{s}}{T-s}\right|&\geq  \frac{(1-\veps)}{T} \sum\limits_{\substack{s\in\varphi^{-1}(\{w\})\\ |\Im s|\leq T}}  \Re s.
 \end{align*}
We can conclude by letting $T$ to $+\infty$ and $\veps$ to $0$.

Regarding b), let $\chi\in\TT^\infty$ and $w\in\CC_0$ with $\Re w\leq c_0.$ Lindel\"{o}f principle says that 
$$\sum_{s\in\psi_\chi^{-1}(\{w\})} \log\left|\frac{\overline{s}+2}{s-2}\right|\leq \log \left|\frac{\overline w+\psi_\chi(2)}{w-\psi_\chi(2)}\right|.$$
Now, when $\varphi_\chi(s)=w,$ then $0<\Re s\leq 1$ since $\Re w\leq c_0$ so that 
$$N_{\psi_\chi}(w)\leq C \sum_{s\in\psi_\chi^{-1}(\{w\})} \log\left|\frac{\overline s+2}{s-2}\right|.$$
Finally it was shown in \cite{BAY21} that 
$$\log \left|\frac{\overline w+\psi_\chi(2)}{w-\psi_\chi(2)}\right|\leq C\frac{\Re w}{1+(\Im w)^2}$$
where $C$ does not depend neither on $\chi\in\TT^\infty$ nor on $w$ with $\Re w\leq c_0.$
\end{proof}

\subsection{A comparison-type principle and a polygonal compactness theorem}
We shall now apply the idea of the previous subsection 
when $\varphi\in\GHZ$ maps $\CC_0$ into a subdomain $D$ of $\Cud.$ 
Lindel\"{o}f principle helps us to find better estimates on $M_\varphi$. Indeed, provided $D$ admits a Green's function, the proof of Theorem \ref{thm:littlewoodinequality}
shows that 
\begin{equation}\label{eq:littlewoodsubdomain}
 M_\varphi(w)\ll g_D(w,\varphi(+\infty)),\ w\in\Cud\backslash\{\varphi(+\infty)\}.
\end{equation}

We deduce the following comparison principle.  Under similar conditions a norm-comparison principle appeared in \cite{BP20}. 
\begin{theorem}\label{thm:comparisonprinciple}
Let $\varphi\in\GHZ$ be such that $\varphi(\CC_0)\subset D\subset\Cud$ where $D$ is a simply connected domain. 
Let $R_D$ be the Riemann map from $\mathbb D$ onto $D$ such that $R_D(0)=\varphi(+\infty)$
and let $\varphi_D=R_D(2^{-s}).$ Assume that $C_{\varphi_D}$ is compact. Then $C_\varphi$ is compact.
Moreover, if $D\subset \{|\Im(s)|\leq C\}$ for some $C>0$ and if $C_{\varphi_D}$ belongs to $S_{2p}$, $p\geq 1,$
then $C_\varphi$ belongs to $S_{2p}$.
\end{theorem}
\begin{proof}
 Let $T=i\frac{2\pi}{\log(2)}.$ By $T$-periodicity of $\varphi_D,$ we know that 
 \begin{align*}
  M_{\varphi_D}(w)&\sim \sum_{\substack{|\Im s|\leq T/2\\s\in\varphi_{D}^{-1}(\{w\})}}\Re s\\
  &\sim \sum_{\substack{ z\in\DD\\ z\in R_D^{-1}(\{w\})}} \log\left(\frac 1{|z|}\right)\\
  &\sim g_D(w,\varphi_D(+\infty))
 \end{align*}
by equality in Lindel\"{o}f principle for a Riemann map. Hence our assumption on $C_{\varphi_D}$ gives
an estimate on $M_{\varphi_D}$ which transfers to $M_\varphi$ thanks to \eqref{eq:littlewoodsubdomain} which
itself gives the corresponding result on $C_\varphi$. Observe that in both cases,
we use the \emph{characterization} of compactness or membership to $S_{2p}$.
\end{proof}
\begin{remark}
 In Theorem \ref{thm:comparisonprinciple}, we can only assume that $D$ admits a Green's function and use
 for $R_D$ a universal covering map of $D$.
\end{remark}

The most interesting case occurs when $\varphi(\CC_0)$ is mapped into an angular sector contained in $\Cud.$ This leads to Theorem \ref{thm:angularsectorzero} that we now prove.
\begin{proof}[Proof of Theorem \ref{thm:angularsectorzero}]
 Green's function of the domain $\Omega=\left\{s\in\Cud:\ \left|\arg(s)-\frac 12\right|<\frac{\pi}{2\alpha}\right\}$ is 
 $$g_\Omega(z,w)=\log\left|\frac{\left(z-\frac{1}{2}\right)^\alpha+\overline{\left(w-\frac{1}{2}\right)}^\alpha}{\left(z-\frac{1}{2}\right)^\alpha-\left(w-\frac{1}{2}\right)^\alpha}\right|.
$$
By \eqref{eq:littlewoodsubdomain} and \cite[Lemma 2.3]{BP21}, for $w\in\varphi(\CC_0)\subset\Omega$, 
\begin{align*}
 M_\varphi(w)&\ll \log\left|\frac{\left(w-\frac{1}{2}\right)^\alpha+\overline{\left(\varphi(+\infty)-\frac{1}{2}\right)}^\alpha}{\left(w-\frac{1}{2}\right)^\alpha-
 \left(\varphi(+\infty)-\frac{1}{2}\right)^\alpha}\right|\\
 &\ll \Re\left(w-\frac 12\right)^\alpha\\
 &\ll \PRewud^ \alpha
\end{align*}
provided $|w-\varphi(+\infty)|>\delta$ for some fixed $\delta>0.$
The proof of compactness follows from the characterization \eqref{eq:compact}, the proof of the Schatten class part follows
from Corollary \ref{cor:carleson}. 
Indeed, letting $\sigma_\infty=2\Re\varphi(+\infty),$ let $T>0$ be such that $|\Im(w)|\leq T$ for all $w\in\varphi(\CC_0)\cap(\Cud\backslash \CC_{\sigma_\infty}).$
Then 
$$\int_{\Cud\backslash \CC_{\sigma_\infty}}\frac{(M_\varphi(w))^p}{\PRewud^{p+2}}(1+|\Im(w)|)^{2(p-1)}dA(w)\ll \int_{\frac 12}^{\sigma_\infty}
\int_{-T}^T \frac{(1+|\Im(w)|)^{2(p-1)}}{\PRewud^{p+2-p\alpha}}dt d\sigma<+\infty$$
since $p>1/(\alpha-1)$. Moreover,
$$\int_{\CC_{\sigma_\infty}}\frac{(M_\varphi(w))^p}{\PRewud^{p+2}}(1+|\Im(w)|)^{2(p-1)}dA(w)\ll \int_{\CC_{\sigma_\infty}} \frac{dA(w)}{(1+|\Im(w)|)^2\PRewud^2}<+\infty$$
by using \eqref{eq:superlittlewood}.
\end{proof}

Using properties of conformal maps, we can extend the compactness part
of Theorem \ref{thm:angularsectorzero} to slightly more general domains.
This is the analogue of the polygonal compactness theorem, \cite{SHAP93}, in our setting.
\begin{theorem}\label{thm:polygonal}
Let $\varphi\in\GHZ$ be such that for some $\delta>0$, the set $\varphi(\CC_0)\bigcap\{\frac{1}{2}<\Re s\leq\frac{1}{2}+\delta\}$ is contained in a finite union of angular sectors $\bigcup_{j=1}^d \left\{\left|\arg\left(z-\frac 12-i\tau_j\right)\right|<\alpha_j\right\}$ with $\tau_j\in\mathbb R$ and $\alpha_j\in(0,\pi/2)$. Then $C_\varphi$ is compact on $\mathcal H^2$.
\end{theorem}
\begin{proof}
Let us consider the Riemann map $f=R_D$, where $D=\mathbb{C}_{\frac{1}{2}+\delta}\cup \bigcup_{j=1}^d \left\{\left|\arg\left(z-\frac 12-i\tau_j\right)\right|<\alpha_j\right\}$ and let $\{w_n\}_{n\geq 1}$ 
be an arbitrary sequence such that $\Re w_n\rightarrow\frac{1}{2}^+$.
Since $M_\varphi(w_n)=0$ if $w_n\notin D,$ we can assume that the sequence converges to a corner boundary point $f(e^{i\theta_0})\neq\infty$. 
 Then, by \eqref{eq:littlewoodsubdomain} and Koebe's quarter theorem \cite[Corollary~1.4]{POM92}
\begin{multline*}
M_\varphi(w_n)\ll g_D(w_n,f(0))\ll 1-\left|f^{-1}(w_n)\right|^2\ll\\ \dist(w_n,\partial D)\left|f'\left(f^{-1}(w_n)\right)\right|^{-1}\leq \left(\Re w_n-\frac{1}{2}\right)\left|f'\left(f^{-1}(w_n)\right)\right|^{-1}.
\end{multline*}
It is sufficient to prove that $\left|f'\left(f^{-1}(w_n)\right)\right|\rightarrow\infty$, as $n\rightarrow\infty$. By the Kellogg-Warschawski theorem \cite[Theorem~3.9]{POM92} and the Carathéodory extension theorem \cite[Chapter~2]{POM92}
\begin{equation*}
\left|f'\left(f^{-1}(w_n)\right)\right|\gg \left|f^{-1}(w_n)-e^{i\theta_0}\right|^{\alpha-1}\rightarrow \infty,
\end{equation*}
where $\alpha=\max\{a_j:1\leq j\leq d\}.$
\end{proof}

\begin{remark}
Our techniques apply also for symbols $\psi=c_0s+\varphi\in\mathfrak{G}_{\geq 1}$. Although, the range of such a symbol cannot meet the imaginary axis in an angular sector or
more generally inside a domain $D$ where $\Im w$ is bounded for $w\in D\bigcap(\mathbb{C}_0 \setminus\mathbb{C}_\varepsilon),\,\varepsilon>0$. 
If that was the case then we would be able to find a point $s(\varepsilon)\in\mathbb{C}_0\setminus\mathbb{C}_{\frac{\varepsilon}{4c_0}}$ such that $\Re\varphi\left(s(\varepsilon)\right)\leq \varepsilon/4$.
The Dirichlet series $\varphi$ converges uniformly in $\mathbb{C}_{\Re s(\varepsilon)/2}$, \cite{BK23}. By almost periodicity we can find an increasing unbounded sequence of positive numbers $\{T_n\}_{n\geq 1}$ such that
$$\Re \varphi(s(\varepsilon)+iT_n))\leq \frac{\varepsilon}{2}$$
so that $ \Re \psi(s(\varepsilon)+iT_n)\leq \frac{3\varepsilon}{4}$. 
We observe that $\left|\Im \psi(s(\varepsilon)+iT_n)\right|\rightarrow +\infty$, this contradicts our assumption.  
\end{remark}

\subsection{On the boundedness on $\mathcal H^p$}
We conclude this section by the proof of Theorem \ref{thm:p<2}.
We will use Hilbertian methods to prove that our assumption implies that $C_\varphi$ is bounded as an operator from $H$ to $\mathcal H^2,$ 
where $H$ is a Hilbert space of Dirichlet series containing $\mathcal H^p.$
To do this, we need another class of Bergman spaces of Dirichlet series,
the spaces $\mathcal{A}_a,\,\alpha\geq 1.$ They are defined as 
\begin{equation*}
\mathcal{A}_\alpha=\left\{f(s)=\sum\limits_{n\geq1}a_n n^{-s}:\norm{f}^2_{\mathcal{A}_\alpha}=\sum\limits_{n\geq1}\frac{|a_n|^2}{d_\alpha(n)}<\infty\right\},
\end{equation*}
where by $d_\alpha(n)$ we denote the coefficients of the Dirichlet series $\left(\zeta(s)\right)^\alpha, s\in\mathbb{C}_1$.
In particular, $d_2(\cdot)$ is the divisor counting function. $\mathcal{A}_\alpha$ is a reproducing kernel Hilbert space,
the reproducing kernel at a point $s_0\in\mathbb{C}_{\frac{1}{2}}$ being the function $\left(\zeta(\overline{s_0}+\cdot)\right)^\alpha$.
The analogue of the embedding theorem for $\mathcal A_\alpha$ reads:
\begin{lemma}[\cite{OL11}]
For every $f\in\mathcal{A}_\alpha$ and every interval $I\subset\mathbb{R}$ there exists a constant $C=C(|I|)$ such that
\begin{equation}\label{eq:ol}
\int\limits_{\frac{1}{2}}^{1}\int\limits_{I}\left|f'(\sigma+it)\right|^2\left(\sigma-\frac{1}{2}\right)^\alpha\,dt\,d\sigma\leq C \norm{f}^2_{\mathcal{A}_\alpha}.
\end{equation}
\end{lemma}  
\begin{proof}[\textbf{Proof of Theorem \ref{thm:p<2}}]
Let us set $\alpha=2k/p.$ Let $\veps>0$, $C>0$ be such that $\Re w\in(0,\veps)$ implies 
$$M_\varphi(w)\leq C\left(\Re w-\frac 12\right)^{\alpha}.$$
Let $T>0$ be such that $\varphi(\CC_0)\cap (\Cud\backslash \CC_{\frac 12+\veps})\subset \left[\frac 12,\frac12+\veps\right]\times[-T,T]$. By \eqref{eq: Stanton},
for $f\in\mathcal H^p\subset\mathcal H^{2k},$
\begin{align*}
\norm{C_\varphi(f)}_{\mathcal{H}^{2k}}^{2k}&=\norm{C_\varphi(f^k)}_{\mathcal H^2}^2\\
&=\left|f^k(\varphi(+\infty))\right|^2+\frac{2}{\pi}\int\limits_{\mathbb{C}_{\frac{1}{2}}}\left|(f^{k})'(w)\right|^2M_{\varphi}(w) \, dA(w)\\
&\ll \int\limits_{\frac{1}{2}}^{\frac{1}{2}+\varepsilon}\int\limits_{-T}^{T}\left|(f^k)'(\sigma+it)\right|^2\left(\sigma-\frac{1}{2}\right)^{\alpha}\,dt\,d\sigma\\
&\quad\quad\quad +
\left|f^k(\varphi(+\infty))\right|^2+\frac{2}{\pi}\int\limits_{\mathbb{C}_{\frac{1}{2}+\varepsilon}}\left|(f^k)'(w)\right|^2M_{\varphi}(w) \, dA(w).
\end{align*}
Let us write $f^k=\sum_{j\geq 1}a_j j^{-s}$. By the Cauchy-Schwarz inequality, for all $w\in\CC_{\frac 12+\veps},$
\begin{align*}
\left|(f^k)'(w)\right|&\leq \sum\limits_{j\geq2}\frac{|a_j|\log j}{j^{\Re w}}\\
&\leq \left(\sum_{j\geq 2}\frac{|a_j|^2}{d_\alpha(j)}\right)^{1/2}\left(\sum_{j\geq 2}\frac{ d_\alpha(j)\log^2 j}{j^{2\Re w}}\right)^{1/2}\\
&\leq C(\veps) |2^{-w}| \|f^k\|_{\mathcal A_\alpha}.
\end{align*}
By the local embedding theorem \eqref{eq:ol}, the boundedness of pointwise evaluation at $\varphi(+\infty)$ and the
continuity of $C_\varphi$ on $\mathcal H^2,$ applied to $2^{-s},$ we get 
$$\|C_\varphi(f)\|_{\mathcal H^{2k}}^{2k}\ll \|f^k\|_{\mathcal A_\alpha}^2.$$
Now, the inclusion operator $i:\mathcal{H}^{p/k}\rightarrow\mathcal{A}_\alpha$ is contractive, \cite{KUL22}. Therefore
$$\|C_\varphi(f)\|_{\mathcal H^p}\leq \|C_\varphi(f)\|_{\mathcal H^{2k}}\ll \|f^k\|_{\mathcal A_\alpha}^{1/k}\leq \|f^k\|^{1/k}_{\mathcal H^{p/k}}=\|f\|_{\mathcal H^p}.$$

Let us turn to compactness. Let $\{f_n\}_{n\geq 1}$ be a sequence of $\mathcal H^q$ converging weakly to $0$.
We set $g_n=f_n^k$ and observe that $(g_n)$ converges pointwise to $0$ on $\Cud$ and that the Dirichlet coefficients
$\widehat g_n(j)$ converge to $0$ for each $j\geq 1.$

We work as above but we now set $\alpha=2k/q$ and consider $\delta\in(0,\veps).$ 
Then 
\begin{align*}
 \norm{C_\varphi(f_n)}_{\mathcal H^q}^{2k}&\leq \norm{C_\varphi(f_n)}_{\mathcal H^{2k}}^{2k}=\norm{C_\varphi(g_n)}^2_{\mathcal H^2}\\
 &\leq |g_n(\varphi(+\infty)|^2+\delta^{\frac 1p-\frac 1q}\int\limits_{\frac{1}{2}}^{\frac{1}{2}+\delta}\int\limits_{-T}^{T}\left|g_n'(\sigma+it)\right|^2\left(\sigma-\frac{1}{2}\right)^{\alpha}\,dt\,d\sigma\\
 &\quad\quad\quad+\frac 2\pi \int_{\mathcal C_{\frac 12+\delta}} \left|g_n'(w)\right|^2M_{\varphi}(w) \, dA(w).
\end{align*}
The first term goes to zero as $n$ tends to $+\infty$ and the second term is as small as we want for every $n$ if we adjust $\delta$ small enough.
Therefore it remains to show that, for a fixed $\delta>0,$ the last terms tends to $0$ as $n$ tends to $+\infty.$
Now, for all $n\geq 1$ and all $w\in \mathbb C_{\frac 12+\delta},$
\begin{align*}
 |g_n'(w)|&\leq \sum_{j\geq 2}\frac{|\widehat{g_n}(j)|\log j}{j^{\Re w}}\\
 &\leq \left(\sum_{j\geq 2}\frac{ d_\alpha(j)\log^2 j}{j^{2\Re w-\frac{\delta}2}}\right)^{1/2}
\left(\sum_{j\geq 2}\frac{|\widehat{g_n}(j)|^2}{d_\alpha(j) j^{\frac\delta2}}\right)^{1/2}\\
&\ll |2^{-w}|\left(\sum_{j=2}^N \frac{|\widehat{g_n}(j)|^2}{d_\alpha(j) j^{\frac\delta2}}+\frac{1}{N^{\frac\delta2}}\|g_n\|_{\mathcal A_\alpha}^2\right)^{1/2}.   
 \end{align*}
Since $(g_n)$ is bounded in $\mathcal A_\alpha$, for any $\eta>0,$ there exists $n_0\in\mathbb N$
such that, $|g_n'(w)|\leq \eta |2^{-w}|.$ We now argue as above to conclude that $(C_\varphi(f_n))$ tends to $0$
in $\mathcal H^q.$
\end{proof}

\begin{remark}
	We choose to work with symbols with range into angular sectors for the sake of simplicity. It will be interesting to know if our techniques can be applied to give other examples of geometric conditions related to the behavior of composition operators on Hardy spaces of Dirichlet series.
\end{remark}


\section{Further discussion}
\subsection{Bergman spaces} We focused on the Hardy space $\mathcal{H}^2$, but we can extend our results 
to Bergman spaces of Dirichlet series $\mathcal{D}_{-a},\,a\geq 0$. The class $\mathfrak{G}$ determines again the bounded composition operators on $\mathcal{D}_{-a},\,a\geq 0$. For Dirichlet series symbols $\varphi\in\mathfrak{G}_0$, the compact composition operators $C_\varphi$ have been characterized, \cite{KP22}, in terms of the weighted counting function
\begin{equation*}
M_{\varphi,1+a}(w)=\lim\limits_{\sigma\rightarrow 0^+}\lim_{T\rightarrow \infty}\frac{\pi}{T}\sum\limits_{\substack{s\in\varphi^{-1}(\{w\})\\
		|\Im s|<T\\
		\sigma<\Re s<\infty}}\left(\Re s\right)^{1+a},\qquad w\neq\varphi(+\infty),
\end{equation*}
and similarly with the Hardy space case, $C_\varphi$ is compact on $\mathcal{D}_{-a},\,a\geq 0$ if and only if
\begin{equation*}
\lim\limits_{\Re w\rightarrow\frac{1}{2}^+}\frac{M_{\varphi,1+a}(w)}{\left(\Re w-\frac{1}{2}\right)^{1+a}}=0.
\end{equation*}
\begin{theorem}\label{bergnecessity}
	Let $\varphi\in\mathfrak{G}_0$ and $p\geq4$. A necessary condition for the composition operator $C_\varphi$ to belong to the class $S_p$ is the following:
\begin{equation}\label{eq:main6}
\int\limits_{\mathbb{C}_\frac{1}{2}}\frac{\left(M_{\varphi,1+a}(w)\right)^{\frac{p}{2}}}{\left(\Re w-\frac{1}{2}\right)^{(a+1)\frac{p}{2}+2}}\,dA(w)<+\infty.
\end{equation}
If we further assume that $\varphi$ has bounded imaginary part, then $C_\varphi$ belongs to
the class $S_p,\,p\geq 2$ if and only if $\varphi$ satisfies \eqref{eq:main6}, and for $p>0$ the condition remains necessary.
\end{theorem}

To prove Theorem~\ref{bergnecessity} one can argue in a similar manner with the Hardy space $\mathcal{H}^2$, using the analogue key ingredients, those are:
The change of variables formula \cite[Theorem~1.2]{KP22}, the Littlewood-type inequality \cite[Proposition~5.4]{KP22}, the weak submean value property \cite[Theorem~4.11]{KP22} and the behavior of reproducing kernels \eqref{rep}.

\subsection{Carleson measures}
E. Saksman and J--F. Olsen \cite{OS12} proved that if $\mu$ is a Carleson measure for $\mathcal{H}^2$,
then it is a  Carleson measure for $H^2(\mathbb{C}_{\frac{1}{2}})$. 
The converse is also true with the extra assumption that $\mu$ has compact support. 

A direct consequence of the local embedding theorem is that a sufficient condition for a measure $\mu$ 
in $\left\{\frac{1}{2}<\Re s<\sigma_\infty\right\}$ to be Carleson for $\mathcal{H}^2$ is 
$\{C(\mu_n,H^2(\Cud))\}_{n\in\mathbb{Z}}\in\ell^1$, where $\mu_n$ is the restriction of $\mu$ on the half-strip $\{s\in\mathbb{C}_{\frac{1}{2}}:n\leq \Im s<n+1\}.$ Indeed,
\begin{align*}
\int\limits_{\mathbb{C}_{\frac{1}{2}}}|f(w)|^2d\mu(w)&\ll\sum\limits_{n\in\mathbb{Z}}\,\,\int\limits_{\mathbb{C}_{\frac{1}{2}}}\left|\frac{f(w)}{w-in}\right|^2d\mu_n(w)\\
&\ll\sum_{n\in\mathbb Z}C(\mu_n) \left\|\frac{f(\cdot)}{\cdot-in}\right\|_{H^2(\Cud)}\\
&\ll \sum_{n\in\mathbb Z}C(\mu_n)\norm{f}_{\mathcal H^2}^2\\
&\ll \norm{f}^2_{\mathcal{H}^2}.
\end{align*}
An example of such a measure is the restriction of $\frac{M_\varphi(w)}{\Re w-\frac{1}{2}}dA(w)$ to $\left\{\frac{1}{2}<\Re s<\frac{1+\Re\varphi(+\infty)}{2}\right\}$.
The above condition is not necessary, as we will exemplify now.

We consider the sequence $\{s_n\}_{n\geq 1}$, where
$$s_n=\frac{1}{2}+\left(\frac{1}{2}\right)^n+i\left(n+\frac{1}{2}\right).$$
As we will prove in a moment the measure $d\mu(w)=\sum\limits_{n\geq 1}(\Re s_n-\frac{1}{2})\delta_{s_n}(w)$ is a Carleson measure for $\mathcal{H}^2$, where $\delta_{s_n}(w)$ is a Dirac mass at $s_n$ . The restriction $\mu_n,\,n\geq 1$ has the form
$$d\mu_n(w)=\left(\Re s_n-\frac{1}{2}\right)\delta_{s_n}(w).$$

Let $Q_n,\,n\geq 1$ be the square with center at the point $s_n$ and one side $I_n$ on the line $\{\Re s=\frac{1}{2}\}$. Then,
$$\mu_n(Q_n)= \frac{\left|I_n\right|}{2},$$
consequently $\{C(\mu_n)\}_{n\in\mathbb{Z}}\notin\ell^1$.
It remains to prove that $d\mu(w)=\sum\limits_{n\geq 1}(\Re s_n-\frac{1}{2})\delta_{s_n}(w)$ is a Carleson measure for $\mathcal{H}^2$. Actually, this is true for every sequence  $\{s_n\}_{n\geq 1}$  in $\mathbb{C}_{\frac{1}{2}}$ such that
\begin{equation}\label{eq:carlesonstrange}
 \Re s_{n+1}-\frac{1}{2}\leq a\left(\Re s_n-\frac{1}{2}\right),\qquad n\in\mathbb{N},
\end{equation}
for some $a\in (0,1)$. 
We follow an argument of \cite[Section 4]{SS61}, see also \cite{B51}.
It is sufficient to prove that the matrix $A=\left[\frac{\zeta(s_i+\overline{s_j})}{\sqrt{\zeta(2\Re s_i)}\sqrt{\zeta(2\Re s_j)}}\right]_{i,j\geq1}$
defines a bounded operator on $\ell^2$. We will prove that for every $j\in\mathbb{N}$
\begin{equation}\label{eq:schurtest}
\sum\limits_{i\geq 1}\frac{|\zeta(s_i+\overline{s_j})|}{\sqrt{\zeta(2\Re s_i)}\sqrt{\zeta(2\Re s_j)}}\leq C,
\end{equation}
the result will then follow from Schur's test \cite[Section~3.3]{ZHU07}.

The Riemann zeta function has a simple pole at $1$. Therefore, \eqref{eq:carlesonstrange} yields the existence of $i_0\geq 1$ and of $b\in(0,1)$
such that, for all $i\geq i_0,$ 
\begin{equation*}
\frac{\zeta(2\Re s_i)}{\zeta(2\Re s_{i+1})}\leq b,
\end{equation*}
where $b\in (0,1)$. We only need to prove \eqref{eq:schurtest} for $j\geq i_0.$ On the one hand,
$$\sum_{1\leq i\leq i_0}\frac{|\zeta\left(s_i+\overline{s_j}\right)|}{\sqrt{\zeta(2\Re s_i)}\sqrt{\zeta(2\Re s_j)}}\leq i_0 \frac{|\zeta\left(\Re s_{i_0}+\Re s_j\right)|}{\sqrt{\zeta(2\Re s_1)}\sqrt{\zeta(2\Re s_j)}}\leq C.$$
On the other hand,
\begin{align*}
\sum\limits_{i\geq 1}\frac{|\zeta\left(s_i+\overline{s_j}\right)|}{\sqrt{\zeta(2\Re s_i)}\sqrt{\zeta(2\Re s_j)}}&\leq \sum\limits_{i\geq 1}\frac{|\zeta(\frac{1}{2}+\max\{\Re s_i,\Re s_j\})|}{\sqrt{\zeta(2\Re s_i)}\sqrt{\zeta(2\Re s_j)}}\\
&\ll \sum\limits_{i_0\leq i\leq j}\sqrt{\frac{\zeta(2\Re s_i)}{\zeta(2\Re s_j)}}+\sum\limits_{i\geq j\geq i_0}\sqrt{\frac{\zeta(2\Re s_j)}{\zeta(2\Re s_i)}}\\
&\ll\sum\limits_{i_0\leq i\leq j}b^{\frac{j-i}{2}} + \sum\limits_{i\geq j\geq i_0}b^{\frac{i-j}{2}}\\
&\leq C.
\end{align*}

\bibliographystyle{amsplain-nodash} 
\bibliography{ref3} 
\end{document}